\documentclass[11pt]{article}
\usepackage{amsmath}
\usepackage{dsfont}
\usepackage{mathrsfs}
\usepackage{amsmath,amssymb}
\usepackage{amsfonts}
\usepackage{hyperref}
\usepackage{amsthm}
\usepackage{graphicx}
\usepackage{subfigure}
\usepackage{xcolor}

\hfuzz=\maxdimen
\theoremstyle{definition}
\newtheorem{lemma}{Lemma}[section]

\newtheorem{theorem}[lemma]{Theorem}
\newtheorem{corollary}[lemma]{Corollary}

\newtheorem{remark}{Remark}

\newtheorem{claim}{Claim}

\numberwithin{equation}{section}

\DeclareFixedFont{\Acknowledgment}{OT1}{cmr}{bx}{n}{14pt}
\begin{document}

\title{\bf Positive Solutions of $p$-th Yamabe Type Equations on Infinite Graphs}
\author{Xiaoxiao Zhang, Aijin Lin}
\date{}
\maketitle

\begin{abstract}
Let $G=(V,E)$ be a connected infinite and locally finite weighted graph, $\Delta_p$ be the $p$-th discrete graph Laplacian. In this paper, we consider the $p$-th Yamabe type equation
$$-\Delta_pu+h|u|^{p-2}u=gu^{\alpha-1}$$
on $G$, where $h$ and $g$ are known, $2<\alpha\leq p$. The prototype of this equation comes from the smooth Yamabe equation on an open manifold. We prove that the above equation has at least one positive solution on $G$.
\end{abstract}
\vspace{12pt}
\let\thefootnote\relax \footnotetext {The first author
is supported by the National Natural Science Foundation of China
(Grant No. 11431003). The second author is supported by the National Natural Science Foundation of China
(Grant No. 11401578).}

%\textbf{Mathematics Subject Classification (2010).} 52C25, 52C26, 53C44.\\
%\tableofcontents
\section{Introduction}

Recently, the investigations of discrete weighted Laplacians and various equations on graphs have attracted much attention (cf. \cite{GLY,GLY2,GLY3,Ge1,Ge2,Ge3,Ge4,ZL}). Grigor'yan, Lin and Yang \cite{GLY3} first studied a Yamabe type equation on a finite graph $G$ as follows
 \begin{equation}\label{YE1}
 -\Delta u+hu=|u|^{\alpha-2}u,\quad \alpha>2
 \end{equation}
  where $\Delta$ is a usual discrete graph Laplacian, and $h$ is a positive function defined on the vertices of $G$. They show that the above equation (\ref{YE1}) always has a positive solution. Inspired by their work, Ge and Jiang \cite{Ge4} studied the following Yamabe type equation on an infinite graph, that is
  \begin{equation}\label{alpha>p}
    -\Delta_pu+hu^{p-1}=gu^{\alpha-1},\ u>0
  \end{equation}
  where $\Delta_{p}$ is $p$-th discrete graph Laplacian.\par
 Now, we recall the main result in \cite{Ge4}.
 \begin{theorem}\label{GJ}(\cite{Ge4}, Theorem 1)
  Consider the $p$-th Yamabe equation (\ref{alpha>p}) on a connected, infinite and locally finite graph $G$ with $\alpha>p\geq2$. Assume $g\geq0$ and $g$ is bounded from above, $h$ satisfies $\inf_{x\in V}h(x)>0$ and $\inf_{x\in V}h(x)\mu(x)>0$. Further assume $h^{-1}\in L^{\delta}(V)$ for some $\delta>0$ (or $h(x)\rightarrow\infty$ when $x\rightarrow\infty$), then (\ref{alpha>p}) has a positive solution.
  \end{theorem}
 From this result, one still needs to know:\\
 \centerline{Can one solve the $p$-th Yamabe equation (\ref{alpha>p}) under the assumption $2<\alpha\leq p$ ?}\par
 The main purpose of this paper is to answer the above question. This paper is organized as follows: In section 2, we give some notations on graph and state main results. Existence of positive solutions in an infinite graph is proved in section 3.

\section{Settings and main results}
First let's recall basic definitions for weighted graphs. Let $G=(V,E)$ be a locally finite graph, where $V$ denotes the vertex set and $E$ denotes the edge set, $\omega:V\times V\ni(x,y)\mapsto\omega_{xy}\in[0,\infty)$ be an edge weight function satisfying
\begin{itemize}
\item $\omega_{xy}=\omega_{yx},\quad\quad \forall x,y \in V,$
\item $\sum_{y\in V}\omega_{xy}<\infty, \quad\quad \forall x\in V,$
\end{itemize}
and $\mu: V\ni x \mapsto \mu(x)\in(0,\infty)$ be a measure on $V$ of full support, and for $x,y\in V, \{x,y\}\in E$ if and only if $\omega_{xy}>0$, in symbols $x\thicksim y$. Alternatively, $\omega_{xy}$ can be considered as a positive function on the set $E$ of edges, that is extended to be $0$ on non-edge pairs $(x,y)$. Note that $G=(V,E)$ possibly possesses self-loops.
We say that a graph is locally finite if for any $x\in V$, there holds $\sum_{y\sim x}1<\infty$. Throughout this paper, we denote $C_{G,h,\cdots}$ as some positive constant depending only on the information of $G,H,\cdots$. Note that the information of $G$ contains $V,E,\mu$ and $\omega$.

For any function $u:V\rightarrow\mathbb{R}$, the $\mu$-Laplacian (or Laplacian for short)
of $u$ is defined as
\begin{equation*}
\Delta u(x)=\frac{1}{\mu(x)}\sum_{y\sim x}w_{xy}(u(y)-u(x)).
\end{equation*}
 Denote $C(V)$ as the set of all real functions defined on $V$, then when $V$ is an infinite (a finite) set, $C(V)$ is an infinite (a finite) dimensional linear space with the usual functions additions and scalar multiplications.
When $p\geq2$,  the genelized discrete graph $p$-Laplacian $\Delta_p:C(V)\rightarrow C(V)$ is defined as
\begin{equation}\label{p-LDef}
\Delta_p f(x)=\frac{1}{\mu (x)}\sum_{y\thicksim x}\omega_{xy}\left|f(y)-f(x)\right|^{p-2}(f(y)-f(x)),
\end{equation}
for $f\in C(V)$ and $x\in V$.
With respect to the vertex weight $\mu$, the integral of $f$ over $V$ is defined by
$$\int_Vfd\mu=\sum\limits_{x\in V}\mu(x)f(x),$$
for any $f\in C(V)$. Set $\mathrm{Vol}(G)=\int_Vd\mu$. Note that $f$ may not be integrable generally. Denote $L^q(V)$ as the space of all $q$-th integrable functions on $V$.

We define a space of functions
 \begin{equation}
 \mathcal{H}=\left\{u\in L^p(V): \int_V \left(|\nabla_p u|^p+h|u|^p\right)d\mu <+\infty\right\}
 \end{equation}
 with a norm
 \begin{equation}\|u\|_{\mathcal{H}}=\left(\int_V\left(|\nabla_p u|^p+h|u|^p\right)d\mu\right)^{1/p},
 \end{equation}
where $|\nabla_p u|$ is defined as
\begin{equation}\label{ptd}
 |\nabla_p u(x)|=\Big(\frac{1}{2\mu(x)}\sum_{y\thicksim x}\omega_{xy}\big|u(y)-u(x)\big|^p\Big)^{1/p},
\end{equation}
which implies the following identity
$$\int_V |\nabla_p u|^p d\mu=\sum_{y\thicksim x}\omega_{xy}\big|u(y)-u(x)\big|^p.$$
Now we state our main result as follows
\begin{theorem}\label{theorem2}
Let $G=(V,E)$ be a connected, locally finite and infinite graph. Assume that its weight satisfies $\omega_{xy}=\omega_{yx}$ for all $y\thicksim x\in V,$ and that its measure $\mu(x)>0$ for all $x\in V$. Suppose $2<\alpha\leq p$, $g\geq0$ and $g$ is bounded from above, $h$ satisfies $\inf_{x\in V}h(x)>0$ and $\inf_{x\in V}h(x)\mu(x)>0$. Further assume $h^{-1}\in L^{\delta}(V)$ for some $0<\delta<\frac{1}{p-2}$, then the
equation
\begin{equation}\label{equation2}
-\Delta_pu+h|u|^{p-2}u=gu^{\alpha-1}
\end{equation}
has a positive solution.
\end{theorem}
\begin{remark}
Grigor'yan, Lin and Yang \cite{GLY3} studied the following equation
$$-\Delta_pu+h|u|^{p-2}u=f(x,u),\ p>1$$
and established some existence results under certain assumptions of $f(x,u)$. However, it is remarkable that the definition of discrect graph $p$-Laplacian $\Delta_p$ in \cite{GLY3} is different from our definition in (\ref{p-LDef}) when $p\neq2$. Moreover, under their assumptions, $f(x,u)$ can not choose $gu^{\alpha-1}$ for $2<\alpha\leq p$.
\end{remark}
Combining Theorem \ref{GJ} with Theorem \ref{theorem2}, we have
\begin{corollary}
  Let $G=(V,E)$ be a connected, locally finite and infinite graph. Assume that its weight satisfies $\omega_{xy}=\omega_{yx}$ for all $y\thicksim x\in V,$ and that its measure $\mu(x)>0$ for all $x\in V$. Suppose $\alpha>2$, $p\geq2$, $g\geq0$ and $g$ is bounded from above, $h$ satisfies $\inf_{x\in V}h(x)>0$ and $\inf_{x\in V}h(x)\mu(x)>0$. Further assume $h^{-1}\in L^{\delta}(V)$ for some $0<\delta<\frac{1}{p-2}$, then the
equation
\begin{equation*}
-\Delta_pu+h|u|^{p-2}u=gu^{\alpha-1}
\end{equation*}
has a positive solution.
\end{corollary}

\section{Proof of Theorem \ref{theorem2}}
For each $u\in \mathcal{H}$, we define a functional
$$J(u)=\int_V (|\nabla_pu|^p+h|u|^p)d\mu.$$
It is easy to see that $J(u)=\|u\|_{\mathcal{H}}^p$ and $J$ is continuously differentiable. For any positive $\theta$, set
\begin{equation*}
G(x,s)=\begin{cases}g(x)\theta s^{\alpha}\ &s\geq0,\\0 & s<0.\end{cases}
\end{equation*}
It is continuously differentiable with respect to $s$ with $\partial_s G(x,s)=\alpha g(x)\theta s^{\alpha-1} $ when $s\geq0$ and $\partial_s G(x,s)=0$ when $s<0$. We write $G'(x, s)=\partial_sG(x,s)$ for short.
Now consider the following functional
\begin{equation}\label{K}
K(u)=\int_V G(x,u)d\mu,\ u\in \mathcal{H}.
\end{equation}

\begin{claim}
The function $K$, defined in (\ref{K}), is continuously differentiable on $\mathcal{H}$ for all $|u|\leq C$, where $C$ is some constant independent of $u$.
\end{claim}
\begin{proof}
By direct calculation, the Fr\'{e}chet derivative of $K(u)$ is a $K'(u)\in \mathcal{H^*}$ with
$$\mathcal{H}\ni v\mapsto K'(u)(v)=\int_V G'(x,u(x))v(x)d\mu.$$
For any $x,y\geq0$, $1\leq a<\infty$, we have an elementary inequality
$$\big|x^a-y^a\big|\leq a|x-y|\big(x^{a-1}+y^{a-1}\big).$$
Note that $\alpha>2$. As $g$ is bounded, by the above elementary inequality, we have
$$|G'(x,u_1(x))-G'(x,u_2(x))|\leq \tilde{C}_{g,\alpha,\theta}|u_1^{\alpha-1}(x)-u_2^{\alpha-1}(x)|\leq C_{g,\alpha,\theta}|u_1(x)-u_2(x)|,$$
where we using the condition $|u|\leq C$ in last inequality.\par
For $\xi\in\mathcal{H}$, it follows
$$|(K'(u_1)-K'(u_2))\xi|\leq C_{g,\alpha,\theta}\int_V|u_1-u_2||\xi|d\mu.$$
Since $h^{-1}\in L^{\delta}(V)$ for some $0<\delta<\frac{1}{p-2}$, there exists some constant $C_{G,h}$, such that
$$\Big(\int_V \frac {1}{h^{\delta}}d\mu\Big)^{\delta}\leq C_{G,h}\ .$$
Denote $\inf_{x\in V}h(x)=C_h>0$. Note that $\frac{1}{p-2}-\delta>0$. For any $x\in V$, we have
\begin{equation}\label{gj}
 \frac{1}{h^{1/(p-2)}(x)}\leq \frac{1}{h^{\delta}(x)}\Big(\frac{1}{\inf_{x\in V}h(x)}\Big)^{\frac{1}{p-2}-\delta}\leq C_{G,\delta,p,h}\frac{1}{h^{\delta}(x)},
\end{equation}
it follows
\begin{equation*}
 \int_V |u_1-u_2|^{\frac{p}{p-1}}d\mu\leq \Big(\int_V \frac{1}{h^{1/(p-2)}}d\mu\Big)^{\frac{p-2}{p-1}}\Big(\int_V h|u_1-u_2|^pd\mu\Big)^{\frac{1}{p-1}}\leq C_{G,\delta,p,h}\|u_1-u_2\|_{\mathcal{H}}^{\frac{p}{p-1}}
\end{equation*}
Moreover,
$$\int_V |u_1-u_2||\xi|d\mu\leq \Big(\int_V |u_1-u_2|^{\frac{p}{p-1}}d\mu\Big)^{\frac{p-1}{p}}\|\xi\|_{\mathcal{H}}
 \ .$$
Thus, we have
$$|(K'(u_1)-K'(u_2))\xi|\leq C_{G,h,p,g,\delta,\theta}\|\xi\|_{\mathcal{H}}\|u_1-u_2\|_{\mathcal{H}} \ .$$
Therefore
$K':\mathcal{H}\rightarrow\mathcal{H^*}$, the Frechet derivative of $K$ satisfies
$$\|K'(u_1)-K'(u_2)\|_{\mathcal{H^*}}\leq C_{G,h,p,g,\delta,\theta}\|u_1-u_2\|_{\mathcal{H}}\ , $$
This implies that $K'$ is continuous on $\mathcal{H}$, i.e. $K$ is continuously differentiable on $\mathcal{H}$.\quad\quad$\Box$
\end{proof}

Now, we consider the functional $J(u)$ under the constraint $K(u)=1$. Since $J(u)\geq0$,
$$\gamma=inf\{J(u): u\in\mathcal{H},\ K(u)=1\}$$
is well defined. Obviously, $\gamma\geq0$. We can also choose a sequence $\{u_i\}_{i\geq1}$ in $\mathcal{H}$ with $J(u_i)\rightarrow\gamma$, $J(u_i)<\gamma+1$ and $K(u_i)=1$.
Denote  $C(G, h)=\inf_{x\in V}h(x)\mu(x)>0$. At each vertex $x\in V$, we have
\begin{equation}\label{BD}
C(G, h)\big|u_i(x)\big|^p\leq h(x)\mu(x)\big|u_i(x)\big|^p\leq \int_V h|u_i|^pd\mu<J(u_i)\leq \gamma+1.
\end{equation}
This means $\big|u_i(x)\big|\leq C_{G,h,p,\gamma}$ for all $x\in V$ and all $i\geq1$. In other words, $\{u_i\}_{i \geq 1}$ are uniformly bounded.
Hence, there exists some $\bar{u}$ such that up to a
subsequence, $u_i \rightarrow \bar{u}$ on $V$. We may well denote this subsequence as $u_i$. Because $G$ is locally finite, $|\nabla_p u_i|\rightarrow|\nabla_p \bar{u}|$ at each vertex $x$.
According to Fatou's lemma, we obtain
\begin{equation}\label{bound}
\int_V(h|\bar{u}|^p+|\nabla_p \bar{u}|^p)d\mu \leq\gamma,
\end{equation}
\begin{equation}\label{2}
K(\bar{u})=\int_V G(x,\bar{u})d\mu\leq 1,
\end{equation}
which implies $\bar{u}\in\mathcal{H}$.

\begin{claim}
$\bar{u}$, as above, is not identically zero on $V$.
\end{claim}
\begin{proof}
Let $x_0 \in V$ be fixed. For any $\epsilon>0$, $h^{-1}\in L^{\delta}(V)$ for some $0<\delta<\frac{1}{p-2}$, there exists some $R>0$ such that
\begin{equation}\label{control}
\Big(\int_{dist(x,x_0)>R} \frac {1}{h^{\delta}}d\mu\Big)^{\delta}\leq {\epsilon}.
\end{equation}
Denote $\inf_{x\in V}h(x)=C_h>0$, and we discuss in two cases. First if $p>\alpha>2$, we gain $\frac{\alpha}{p-\alpha}-\delta>0$. Hence, for any $x\in V$
$$\frac{1}{h^{{\alpha}/(p-\alpha)}(x)}\leq \frac{1}{h^{\delta}(x)}\Big(\frac{1}{\inf_{x\in V}h(x)}\Big)^{\frac{\alpha}{p-\alpha}-\delta'}\leq C_{\alpha,\delta,p,h}\frac{1}{h^{\delta}(x)}.$$
As $g$ is bounded from above and  $\|u_i\|_{\mathcal{H}}=J(u_i)^{\frac{1}{p}}$
, using the definition of $G(x,s)$ and the inequality ((\ref{BD})), we obtain
\begin{align*}
\int_{dist(x,x_0)>R}G(x,u_i)d\mu\leq& C_{\alpha,g,\theta}\int_{dist(x,x_0)>R,u_i(x)>0}u_i^{\alpha}d\mu
\\ \leq& C_{\alpha,g,\theta} \Big(\int_{dist(x,x_0)>R,u_i(x)>0}\frac{1}{h^{{\alpha}/(p-\alpha)}}d\mu\Big)^{\frac{p-\alpha}{\alpha}}
\|u_i\|^{\alpha}_{\mathcal{H}}
\\ \leq&C_{\alpha,g,\delta,p,h,\theta}\ {\epsilon}^{\frac{(p-\alpha)\delta}{\alpha}} \|u_i\|^{\alpha}_{\mathcal{H}}
\\ \leq&C_{\alpha,g,\delta,p,h,\theta}\ {\epsilon}^{\frac{(p-\alpha)\delta}{\alpha}}(\gamma+1)^{\frac{\alpha}{p}}.
\end{align*}
Using $K(u_i)=1$, for $p>\alpha>2$, we have
\begin{equation}\label{p>}
  \int_{dist(x,x_0)\leq R}G(x,u_i)d\mu\geq1-C_{\alpha,g,\delta,p,h,\theta}\ {\epsilon}^{\frac{p-\alpha}{\alpha\delta}}  (\gamma+1)^{\frac{\alpha}{p}}.
\end{equation}
Then let $i\rightarrow\infty$, and we have
$$K(\bar{u})=\int_VG(x,\bar{u})d\mu\geq \int_{dist(x,x_0)\leq R}G(x,\bar{u})d\mu\geq1-C_{\alpha,g,\delta,p,h,\theta}\ {\epsilon}^{\frac{p-\alpha}{\alpha\delta}}  (\gamma+1)^{\frac{\alpha}{p}}.$$
Let $\epsilon\rightarrow0$, we obtain $K(\bar{u})\geq1$.\par
Then in the case of $p=\alpha$, because $\{u_i\}_{i \geq 1}$ are uniformly bounded, using (\ref{gj}), we get
\begin{align*}
  \int_{dist(x,x_0)>R}G(x,u_i)d\mu\leq& C_{g,\theta}\int_{dist(x,x_0)>R,u_i(x)>0}u_i^pd\mu\\ \leq&C_{p,g,\theta}\int_{dist(x,x_0)>R,u_i(x)>0}u_i^{\frac{p}{p-1}}d\mu
\end{align*}
Moreover, using (\ref{gj}), (\ref{BD}) and (\ref{control}), we obtain
\begin{align*}
  \int_{dist(x,x_0)>R} u_i^{\frac{p}{p-1}}d\mu\leq& \Big(\int_{dist(x,x_0)>R}\frac{1}{h^{1/(p-2)}}d\mu\Big)^{\frac{p-2}{p-1}}
  \Big(\int_{dist(x,x_0)>R}hu_i^p d\mu\Big)^{\frac{1}{p-1}}\\ \leq&
   C_{G,\delta,p,h}{\epsilon}^{\frac{(p-2)\delta}{p-1}}\|u_i\|_{\mathcal{H}}^{\frac{p}{p-1}}\\ \leq&
   C_{G,\delta,p,h}{\epsilon}^{\frac{(p-2)\delta}{p-1}}(\gamma+1)^{\frac{1}{p-1}}.
\end{align*}
Then, for $p=\alpha>2$, using $K(u_i)=1$, we have
\begin{equation}\label{p=}
  \int_{dist(x,x_0)\leq R}G(x,u_i)d\mu\geq1- C_{G,\delta,p,h}{\epsilon}^{\frac{(p-2)\delta}{p-1}}(\gamma+1)^{\frac{1}{p-1}}.
\end{equation}
Let $i\rightarrow\infty$ again, we have
$$K(\bar{u})=\int_VG(x,\bar{u})d\mu\geq \int_{dist(x,x_0)\leq R}G(x,\bar{u})d\mu\geq1- C_{G,\delta,p,h}{\epsilon}^{\frac{(p-2)\delta}{p-1}}(\gamma+1)^{\frac{1}{p-1}}.$$
Let $\epsilon\rightarrow0$, we obtain $K(\bar{u})\geq1$.  In a word, we obtain $K(\bar{u})\geq1$ for $p\geq\alpha>2$. By (\ref{2}), we get $K(\bar{u})=1$, which implies that $\bar{u}$ is not identically zero.\quad\quad$\Box$
\end{proof}

\begin{claim}\label{positive}
$\bar{u}$, as above, is positive everywhere on $V$.
\end{claim}
\begin{proof}
From (\ref{bound}), we can prove $|\bar{u}|\leq C_{G,h,p,\gamma}$ which is totally similar to prove $|u_{i}(x)|\leq C_{G,h,p,\gamma}$. Then $\big|\bar{u}+t\varphi\big|\leq C_{G,h,p,\gamma}$ as $t\rightarrow0$, for any $\varphi\in \mathcal{H}$.
 We calculate the Euler-Lagrange equation at $\bar{u}$ under the constraint condition $K(\bar{u})=1$. For any $\varphi\in \mathcal{H}$, there holds
\begin{align*}
0=&\frac{d}{dt}\Big|_{t=0}\Big\{J(\bar{u}+t\varphi)-\lambda\Big(\int_V G(x,\bar{u}+t\varphi)d\mu-1\Big)\Big\}\\=&\int_V \Big(-p\Delta_p\bar{u}+ph|\bar{u}|^{p-2}\bar{u}-\lambda G'(x,\bar{u})\Big)\varphi d\mu.
\end{align*}
Hence, we get
\begin{equation}\label{EL}
-p\Delta_p\bar{u}+ph|\bar{u}|^{p-2}\bar{u}=\lambda G'(x,\bar{u}).
\end{equation}
Noting that
$$\int_V (-\bar{u}\Delta_p\bar{u})d\mu=\int_V |\nabla_p u|^p d\mu,$$
 multiplying $\bar{u}$ on both sides of the equation (\ref{EL}), and taking integration, we can see $\lambda>0$. If $\bar{u}(x)<0$, at some vertex $x\in V$, then by the equation $(\ref{EL})$, we see
 $$\Delta_p \bar{u}(x)<0.$$
However, by the definition of $\Delta_p$, there is a $y\thicksim x$ with $\bar{u}(y)<\bar{u}(x)<0$. In view of the connectedness of the graph $G=(V,E)$, by inductively, we obtain a sequence $x=x_1\thicksim x_2\thicksim x_3\thicksim\cdots$, such that $$\bar{u}(x_i)<\bar{u}(x_{i-1})<\cdots<\bar{u}(x_1)<0.$$
This is impossible because $\mu(x)\geq\mu_{min}>0$, $\inf_{x\in V}h(x)>0$, and $h|\hat{u}|^p$ is integrable. Hence $\bar{u}$ is nonnegative on $V$.
If $\bar{u}$ is not positive everywhere on $V$, we can always find two vertices $x,y$ with $y\thicksim x$, $\bar{u}(x)=0$, $\bar{u}(y)>0.$ Then it follows $\Delta_p \bar{u}(x)>0$ by the definition of $\Delta_p$, which contradicts to the equation (\ref{EL}). Hence $\bar{u}$ is positive everywhere on $V$.\quad\quad$\Box$
\end{proof}

\begin{claim}
The equation (\ref{equation2}) has a strictly positive solution.
\end{claim}
By Claim $1,2,3$, we know that $\bar{u}$, as above, is positive everywhere on $V$, and it satisfies
\begin{equation}\label{change2}
-p\Delta_p\bar{u}+ph{\bar{u}}^{p-1}=\lambda\alpha\theta g \bar{u}^{\alpha-1}.
\end{equation}
If $p>\alpha>2$, choose $\theta=1$ and set $$u=\Big(\frac{p}{\alpha\lambda}\Big)^{\frac{1}{p-\alpha}}\bar{u}$$ in (\ref{change2}), then we have
\begin{equation}\label{Y-E}
 -\Delta_pu+hu^{p-1}=gu^{\alpha-1}.
\end{equation}
If $p=\alpha$, choose $\theta=\frac{1}{\lambda}$ and set $u=\bar{u}$ in  (\ref{change2}), we also obtain (\ref{Y-E}).
This implies that $u$ is a positive solution to the $p$-th Yamabe equation (\ref{equation2}).$\hfill\Box$

\noindent \textbf{Acknowledgements:} The first author would like to thank Professor Yanxun Chang for constant guidance and encouragement. The second author would like to thank Professor Gang Tian and Huijun Fan for constant encouragement and support. Both authors would also like to thank Professor Huabin Ge for many helpful conversations. The first author is supported by National Natural Science Foundation of China under Grant No. 11431003. The second author is supported by National Natural Science Foundation of China under Grant No. 11401578.

Xiaoxiao Zhang: xiaoxiaozhang0408@bjtu.edu.cn\\
Institute of Mathematics, Beijing Jiaotong University, Beijing 100044, P. R. China\\
E-mail address: xiaoxiaozhang0408@bjtu.edu.cn\\
\\
Aijin Lin: aijinlin@pku.edu.cn\\
College of Science, National University of Defense Technology, Changsha 410073, P. R. China\\
E-mail address: aijinlin@pku.edu.cn\\
\end{document}